\documentclass{article}

\usepackage[
	a4paper,
	margin=25mm
]{geometry}
\usepackage{setspace}
\usepackage{sectsty} 
\usepackage{titlesec}

\usepackage{amsmath, amssymb, amsthm}
\usepackage{mathtools}

\usepackage[
	setpagesize=false,
	colorlinks=true,
	linkcolor=blue,
	pdfencoding=auto,
	psdextra,
]{hyperref}
\usepackage{cleveref}
\usepackage[
    sorting=nyt,
    maxbibnames=99
]{biblatex}

\usepackage{etoolbox} 
\usepackage[shortlabels]{enumitem}
\usepackage{xparse} 

\usepackage{graphicx}
\usepackage[dvipsnames]{xcolor}
\usepackage{tikz}

\usepackage{todonotes}
\usepackage[
    deletedmarkup=sout,
    commentmarkup=uwave,
    authormarkuptext=name
]{changes}

\usepackage{comment}

\setlist[enumerate,1]{label=(\arabic*), ref=(\arabic*)}
\setlist[enumerate,3]{label=(\roman*), ref=(\roman*)}

\allsectionsfont{\boldmath} 

\titlelabel{\thetitle.\quad}

\theoremstyle{plain}
\newtheorem{theorem}{Theorem}[section]
\newtheorem{lemma}[theorem]{Lemma}

\newtheorem{conjecture}[theorem]{Conjecture}

\newtheorem{claim}{Claim}[theorem]
\newtheorem*{claim*}{Claim}

\makeatletter
\newenvironment{claimproof}[1][Proof]{\par
	\pushQED{\qed}%
	
	\normalfont \topsep6\p@\@plus6\p@\relax
	\trivlist
	\item[\hskip\labelsep
	\textit{#1}\@addpunct{.}~]\ignorespaces
}{%
	\popQED\endtrivlist\@endpefalse
}
\makeatother

\newlist{Cases}{enumerate}{3}
\setlist[Cases]{parsep=0pt plus 1pt}
\setlist[Cases,1]{wide=0pt, listparindent=\parindent,
    label = \textbf{Case~\arabic*:}, ref = \arabic*}
\setlist[Cases,2]{wide=\parindent, listparindent=\parindent,
    label = \textbf{Case~\arabic{Casesi}-\arabic{Casesii}:}}

\crefname{Casesi}{case}{cases}

\newcounter{case}
\AtBeginEnvironment{proof}{\setcounter{case}{0}}

\crefname{case}{case}{cases}

\theoremstyle{definition}
\newtheorem{definition}[theorem]{Definition}

\newcommand{\calP}{\mathcal{P}}

\newcommand{\calT}{\mathcal{T}}

\newcommand{\calY}{\mathcal{Y}}
\newcommand{\calZ}{\mathcal{Z}}

\newcommand{\ve}{\varepsilon}

\makeatletter
\NewDocumentCommand{\xsideset}{mmme{_^}}{%
  \mathop{%
    \settowidth{\dimen0}{$\m@th\displaystyle#3$}%
    \dimen0=.5\dimen0
    \settowidth{\dimen2}{$%
      \m@th\displaystyle#3%
      \IfValueT{#4}{_{#4}}%
      \IfValueT{#5}{^{#5}}%
    $}%
    \dimen2=.5\dimen2
    \advance\dimen2 -\dimen0
    \sbox6{\scriptspace\z@$\displaystyle{\vphantom{#3}}#1$}
    \sbox8{\scriptspace\z@$\displaystyle{\vphantom{#3}}#2$}
    \ifdim\wd6>\dimen2 \kern\dimexpr\wd6-\dimen2\relax\fi
    {%
     \mathop{\llap{\copy6}{\displaystyle#3}\rlap{\copy8}}\limits
     \IfValueT{#4}{_{#4}}%
     \IfValueT{#5}{^{#5}}%
    }%
    \ifdim\wd8>\dimen2 \kern\dimexpr\wd8-\dimen2\relax\fi
  }%
}
\makeatother

\newcommand{\defeq}{\coloneqq}

\let\originalleft\left
\let\originalright\right
\renewcommand{\left}{\mathopen{}\mathclose\bgroup\originalleft}
\renewcommand{\right}{\aftergroup\egroup\originalright}

\makeatletter
\newcases{lrcases*}
  {\quad}
  {$\m@th{##}$\hfil}
  {{##}\hfil}
  {\lbrace}
  {\rbrace}
\makeatother

\usetikzlibrary{matrix,arrows,decorations.pathmorphing}
\usetikzlibrary{positioning}
\usetikzlibrary{graphs} 
\usetikzlibrary{graphs.standard}

\title{Spanning subdivisions in dense digraphs}

\author{
Hyunwoo Lee%
        \thanks{Department of Mathematical Sciences, KAIST, South Korea and Extremal Combinatorics and Probability Group
(ECOPRO), Institute for Basic Science (IBS).
        E-mail: {\ttfamily hyunwoo.lee@kaist.ac.kr.} Supported by the National Research Foundation of Korea (NRF) grant funded by the Korea government(MSIT) No. RS-2023-00210430, and the Institute for Basic Science (IBS-R029-C4).}
}

\begin{document}
\maketitle

\begin{abstract}
    We prove that an $n$-vertex digraph $D$ with minimum semi-degree at least $\left(\frac{1}{2} + \varepsilon \right)n$ and $n \geq C m$ contains a subdivision of all $m$-arc digraphs without isolated vertices. Here, $C$ is a constant only depending on $\varepsilon.$ This is the best possible and settles a conjecture raised by Pavez-Sign\'{e}~\cite{PavezSigne2023SpanningSI} in a stronger form.
\end{abstract}

\section{Introduction}\label{sec:intro}

Embedding a spanning sparse (di)graph into a dense (di)graph is one of the central topics in extremal combinatorics and has been extensively studied over decades.
Arguably, the most natural condition is the minimum degree condition. For example, a classical result of Dirac~\cite{dirac} asserts that for every positive integer $n \geq 3$, any $n$-vertex graph with minimum degree at least $\frac{n}{2}$ contains a Hamiltonian cycle. 

Inspired by Dirac's theorem, there were various problems and results regarding the minimum degree threshold for the existence of certain spanning structures. One direction is for perfect tilings problem. For graphs $G$ and $H$, we say $G$ has a \emph{perfect $H$-tiling} if $G$ contains a collection of vertex disjoint copies of $H$ that covers every vertex of $G.$ For a complete graph $K_r$ on $r$ vertices, Hajnal and Szemer\'{e}di~\cite{hajnal1970proof} proved that every $n$-vertex graph $G$ with minimum degree at least $\frac{r-1}{r}n$, where $n$ is divisible by $r$, contains a perfect $K_r$-tiling. For a general graph $H$, K\"{u}hn and Osthus~\cite{kuhn2009minimum} determined the exact minimum degree threshold for the existence of perfect $H$-tiling up to an additive constant.  

Since then, many variants of the K\"{u}hn-Osthus theorem on graph tilings have been considered (for instance, see \cite{balogh2022tilings,han2021tilings,hyde2019degree,lo2015f}). Recently, the author~\cite{lee2023perfect} considered a new variant of the tiling problem: subdivision tilings. we say that $H'$ is a \emph{subdivision} of $H$ if $H'$ is obtained by replacing the edges of $H$ with vertex-disjoint paths. For graphs $G$ and $H$ we say that $G$ has a \emph{perfect $H$-subdivision tiling} if $G$ contains a collection of vertex-disjoint copies of subdivisions of $H$ that covers all vertices of $G.$ The asymptotic behavior of the tight minimum degree threshold for the existence of perfect $H$-subdivision tilings was determined for all graphs $H$ in~\cite{lee2023perfect}.

We note that indeed, Dirac's theorem determined the minimum degree threshold for the existence of one spanning $K_3$-subdivision. On the other hand, for a graph $H$, \cite{lee2023perfect} determined the minimum degree threshold for the existence of a spanning graph consisting of vertex-disjoint copies of subdivisions of $H.$ Comparing these two results, it is very natural to ask for the minimum degree threshold for the existence of a spanning subdivision of general graphs $H.$ Very recently, Pavez-Sign\'{e}~\cite{PavezSigne2023SpanningSI} proved that for every $\ve > 0$, there is a constant $C$ such that every graph $G$ on $n \geq Ck^2$ vertices with minimum degree at least $\left(\frac{1}{2} + \ve \right)n$ contains a spanning subdivision of $K_k$ for any positive integer $k.$ More generally, they showed that if $H$ is a $k$-vertex $d$-regular graph with $\log k \leq d \leq k$, then the same result holds for a spanning subdivision of $H$ as long as $n \geq Cdk.$ For general graphs, they proposed the following natural conjecture.

\begin{conjecture}[\cite{PavezSigne2023SpanningSI}]\label{conj:conj}
    For every $\ve > 0$, there is a constant $C$ such that the following holds. For every graph $G$ on $n \geq Cm$ vertices, if the minimum degree of $G$ is at least $\left(\frac{1}{2} + \ve \right)n$, then $G$ contains a spanning subdivision of any $m$-edge graph $H$ with no isolated vertices.
\end{conjecture}

Note that if $H$ has arbitrarily many isolated vertices, then whatever the size of $G$ is, we cannot embed a spanning subdivision of $H.$ Thus the condition, $H$ does not have isolated vertices, is natural. Moreover, $n \geq Cm$ is necessary. For instance, let $G$ be a balanced complete $3$-partite graph on $n = \frac{t^2}{10}$ vertices and $H$ is a clique on $t$ vertices. Then the minimum degree of $G$ is at least $\frac{2n}{3} - 1.$ Since $G$ does not contain a $K_4$, thus if $G$ contains a spanning subdivision $H'$ of $H$, then $H'$ should be a $K_4$-free. By Tur\'{a}n's theorem, at least $\binom{t}{2} - \frac{t^2}{3} > \frac{t^2}{10}$ edges of $H$ should be subdivided. This implies that $|V(H')| > \frac{t^2}{10} = |V(G)|.$ Thus $G$ does not contain a spanning subdivision of $H.$

In this article, we prove \Cref{conj:conj} in a stronger form considering digraphs. Let $H$ be a digraph. We say a digraph $H'$ is a subdivision of $H$ if $H'$ is obtained by replacing arcs of $H'$ with vertex-disjoint directed paths with the same consistent orientation. For a digraph $D$ and $v\in D$, the \emph{semi-degree} of $v$ in $D$ is the minimum of the out-degree of $v$ and in-degree of $v$ in $D.$

\begin{theorem}\label{thm:main}
    Let $\ve > 0$ be a positive real number. Then there is a constant $C > 0$ such that the following holds.
    Let $D$ be a digraph on $n \geq Cm$ vertices with minimum semi-degree at least $\left(\frac{1}{2} + \ve \right)n.$ Then $D$ contains a spanning subdivision of any $m$-arc digraph $H$ with no isolated vertices.
\end{theorem}

By replacing edges of an $n$-vertex graph $G$ with minimum degree at least $\left(\frac{1}{2} + \ve \right)n$ with two arcs in both directions, it is straightforward to check that \Cref{thm:main} implies \Cref{conj:conj}.

The minimum semi-degree condition in \Cref{thm:main} is asymptotically best possible for large $n$ compared with $m.$ Let $H$ be an $m$-arc digraph on $k$ vertices without isolated vertices. We note that $k \leq 2m.$ Let $A$ and $B$ be disjoint sets with $|A| = \lfloor\frac{n}{2} \rfloor - (m + k)$ and $|B| = n - |A|.$ Let $D$ be a digraph on the vertex set $A\cup B$ such that the arc set of $D$ is $\{\overrightarrow{uv}, \overrightarrow{vu}: (u, v)\in A\times B\}.$ Then the minimum semi-degree of $D$ is at least the size of $A$, that is $\lfloor \frac{n}{2}\rfloor - (m + k).$ We now claim that $D$ does not contains a spanning subdivision of $H.$ Assume $D$ has a spanning subdivision of $H.$ Then at least $n - k$ vertices of $D$ are covered by at most $m$ pairwise vertex disjoint directed paths. Since the underlying graph of $D$ is a bipartite graph on the bipartition $(A, B)$, the inequality $||A| - |B|| \leq m + k$ must hold, a contradiction. Thus, $D$ does not contain a spanning subdivision of $H.$

Since a directed Hamiltonian cycle is a spanning subdivision of a directed $2$-cycle, we can consider \Cref{thm:digraph-dirac} as the asymptotic generalization of the following classical digraph analog to Dirac's theorem.

\begin{theorem}[Ghouila-Houri~\cite{ghouilahouri1960condition}]\label{thm:digraph-dirac}
    Let $D$ be an $n$-vertex digraph with minimum semi-degree at least $\frac{n}{2}.$ Then $D$ contains a directed Hamiltonian cycle.
\end{theorem}

Inspired by \Cref{thm:digraph-dirac}, embedding (arbitrarily) oriented Hamiltonian cycles and tree-like structures were extensively studied~\cite{antidirected,arbitrary-oriented,tree-like}.

The proof of the main theorem uses the \emph{absorption method} popularized by R\"{o}dl, Ruci\'{n}ski, and Szemer\'{e}di~\cite{rodl2006dirac} and implicitly used already earlier in~\cite{erdHos1991vertex,krivelevich1997triangle}. The absorption method is an extremely useful tool for finding spanning structures. In order to get more information about the absorption method, we recommend seeing the survey~\cite{zhao2016recent}.

\section{Notations}\label{sec:prelim}
We write $[n] = \{1, 2, \dots , n\}.$ For a statement regarding some parameters $\beta, \alpha_1, \dots, \alpha_t$, we say that it holds when $0 < \beta \ll \alpha_1, \dots, \alpha_t < 1$ if there exists a non-decreasing function $f$ such that for every $0 < \beta \leq f(\alpha_1, \dots, \alpha_t)$ the statement holds. In this article, we will not compute these functions explicitly.

For a digraph $D$, we let $V(D)$ be the vertex set of $D$ and $A(D)$ be the arc set of $D.$
For a vertex $v$ in a directed graph $D$, we denote by $N^+_D(v)$ and $N^-_D(v)$ the out-neighborhood and in-neighborhood of $v$, respectively. We write $d^+_D(v)$ and $d^-_D(v)$ as the size of the out-neighborhood and in-neighborhood of $v$ and we call them out-degree and in-degree of $v$, respectively. We denote by $\delta^0(D)$ the minimum semi-degree, that is $\min\{\delta^+(D), \delta^-(D)\}.$

For a vertex subset $X$ and $Y$, we let $A_D[X, Y]$ be the set of arcs starting from $X$ and ending at $Y.$ We note that $X$ and $Y$ do not need to be disjoint. For a pair of vertices $(u, v) \in V(D)\times V(D)$, we define $N_D(u, v) \defeq \{w\in V(D): \overrightarrow{uw}, \overrightarrow{wv} \in A(D)\}.$ We denote by $d_D(u, v)$ the number $|N_D(u, v)|.$

For a digraph $D$ and vertex subsets $X, Y\subseteq V(D)$, let $D - X + Y$ be a digraph that is an induced digraph of $D$ on the vertex subset $(V(D)\setminus X)\cup Y.$

Let $X$ be a set and $t$ be a positive integer. We write $X^t$ be the $t$-fold cartesian product of $X$, that is $X\times \cdots \times X.$ We say two elements $(y_1, \dots, y_t), (z_1, \dots, z_t)$ of $X^t$ are tuple-disjoint if $\{y_1, \dots, y_t\}\cap \{z_1, \dots, z_t\} = \emptyset.$

\section{$(n, t ,d )$-tuple systems}\label{sec:ntd-system}
In this section, we will prove the most technical lemma. Our proof strategy is applying the absorption method in the digraph, we need to build a suitable collection of structures that can be used in the absorbing process. Since our host structure is the digraph, we need to consider the direction of the structure, thus we consider the following definition, which can handle directed structures. 

\begin{definition}
    Let $n$ and $t$ be two positive integers and let $d$ be a real number in $(0, 1).$ For two sets $X$ and $Y$ with $|X| \leq n^2$ and $|Y| = n$, we say a collection of pairs $\calP \subseteq X\times Y^t$ is an \emph{$(n, t, d)$-tuple system} on $(X, Y)$ if $\calP$ satisfies the following:
    For each element $x \in X$, the size of the set $\{Z\in Y^t: (x, Z)\in \calP\}$ is at least $dn^t.$
\end{definition}
We consider $(x, Z)\in X\times Y^t$ as a pair not a $(t+1)$-tuple. The collection $\calP$ will be used to construct a connector and absorbing path in the proof of \Cref{thm:main}.

Let $\calP$ be the $(n, t, d)$-tuple system on $(X, Y).$ Then by the definition of $(n, t, d)$-tuple system, every elememt $x\in X$ has many $t$-tuples $Z\in Y^t$ such that $(x, Z)\in \calP.$ Thus if we choose a collection $\calY$ of vertex-disjoint copies of members in $Y^t$ carefully, then we may expect that each $x\in X$ still have sufficiently many members $Z\in \calY$ such that $(x, Z)\in \calP.$
The following lemma, which is the key lemma for the proof of \Cref{thm:main}, states that this intuition is true.

\begin{lemma}\label{lem:ntd-absorbing}
    Let
    $$0 < \frac{1}{n}, \beta \ll d, \alpha, \frac{1}{t} \leq 1$$ where $n$ and $t$ are integers.

    Let $\calP$ be an $(n, t, d)$-tuple system on the pair of sets $(X, Y)$ with  $|X| \leq n^2$ and $|Y| = n.$ Then there is a collection of subsets $\calY \subseteq Y^t$ that satisfies the following properties:
    \begin{enumerate}
        \item[$\bullet$] Elements in $\calY$ are pairwise tuple-disjoint,
        \item[$\bullet$] the size of $\calY$ is at most $\frac{\alpha}{t}n,$
        \item[$\bullet$] for every $x\in X$, the number of elements $Z \in \calY$ where $(x, Z)\in \calP$ is at least $\beta n.$
    \end{enumerate}
\end{lemma}

Indeed, the similar notion with $(n, t, d)$-tuple systems and \Cref{lem:ntd-absorbing} were already considered in~\cite{lee2023perfect} in the notion of $(n, t, d)$-systems. The proof of \Cref{lem:ntd-absorbing} is almost the same as the proof of Lemma 4.9 in~\cite{lee2023perfect}. However, for completeness, we write the proof of \Cref{lem:ntd-absorbing} below. 

\begin{proof}[Proof of \Cref{lem:ntd-absorbing}]
    Let fix a constant $c = \min\{\frac{\alpha}{2}, \frac{d}{4t^2}\}.$ We now consider $\calT$ which is a random subset of $Y^t$, such that each element in $Y^t$ is selected independently at random with the probability $p = \frac{c}{n^{t - 1}}.$ Then we observe that the random variable $|\calT|$ follows a binomial distribution and the inequality $\mathbb{E}[|\calT|] \leq \frac{\alpha}{2t}n$ holds. Then by Chebyshev's inequality, with probability $1 - o(1),$ the following holds. 
    \begin{equation}\label{eq:1}
        |\calT| \leq \frac{\alpha}{t}n.
    \end{equation}

    For each $x\in X$, we denote by $\calZ_x$ the set $\{Z\in Y^t: (x, Z)\in \calP\}$. Then the size $|\calT \cap \calZ_x|$ is a random variable that follows a binomial distribution and we note that we have $\mathbb{E}[|\calT \cap \calZ_x|] \geq |\calZ_x|p \geq cd n.$ By Chernoff bound, for each $x\in X$, the size of the random set $\calT \cap \calZ_x$ is at least $\frac{cd}{2}n$ with probability $1 - o(n^{-3}).$ Since $|X| \leq n^2$, we have the following with probability $1 - o(1).$
    \begin{equation}\label{eq:2}
        |\calT \cap \calZ_x| \geq \frac{cd}{2}n \text{ for all } x\in Z.
    \end{equation}
    
    We now count the number of not tuple-disjoint pairs in $\calT.$ Let $I$ be a random variable which is the number of not tuple-disjoint pairs of $\calT.$ We note that we have $$|\{\{A, B\}: A, B\subseteq [n]^t,\text{ } A \text{ and }B \text{ is not a tuple-disjoint} \}| \leq t^2 n^{2t - 1}$$ and each $A$, $B$ are chosen to be in $\calT$ with probability $p = \frac{c}{n^{t-1}}.$ Thus, we have $\mathbb{E}[I] \leq c^2t^2n.$ By Markov's inequality, with probability at least $\frac{1}{2}$,
    \begin{equation}\label{eq:3}
        |I| \leq 2c^2t^2n.
    \end{equation}

    By the union bound, for all sufficiently large $n$, there is a family $\calT$ satisfies all \Cref{eq:1,eq:2,eq:3}.
    Let $\calY$ be a collection obtained from $\calT$ by removing not tuple-disjoint pairs from $\calT.$

    Then $\calY$ is a collection of pairwise tuple-disjoint sets and by \eqref{eq:1}, we have $|\calY| \leq |\calT| \leq \frac{\alpha}{t}n.$ By \eqref{eq:2} and \eqref{eq:3}, for all $x\in X$, the inequality $|\calY \cap \calZ_x| \geq \frac{cd}{2}n - 2c^2t^2n \geq \frac{\alpha d^2}{16t^2}n$ holds. We now set $\beta = \frac{\alpha d^2}{16t^2}.$ Then $\calY$ is the desired collection. This completes the proof. 
\end{proof}

We want to remark that the notion of $(n, t, d)$-systems, $(n, t, d)$-tuple systems, and \Cref{lem:ntd-absorbing} provide systematic frameworks for the absorption method that appeared in~\cite{rodl2006dirac}. We also remark that the proof of \Cref{lem:ntd-absorbing} is almost the same as the proof of Lemma 2.3 from~\cite{rodl2006dirac}.


\section{Proof of \Cref{thm:main}}

In this section, we will prove \Cref{thm:main}. Our strategy is the following. We first construct an `absorbing path' $A$, a directed path of short length such that for every small number of vertex set $U$ disjoint from $A$, there is another directed path $P$ with $V(P) = V(A) \cup U$ and the endpoints of $P$ are the same as $A.$ Thus, $A$ can absorb the vertices of $U$ in some sense. The second step is to choose a small vertex set, `connector', which is disjoint from the absorbing path such that we can connect two arbitrary vertices of the given digraph with a directed path of length two by using one vertex from the connector.

Once we obtain the absorbing path and the connector, the rest is simple.
We embed the branch vertices into arbitrary vertices disjoint from the absorbing path and the connector. Since the size of the absorbing path and the connector is small, the minimum semi-degree of the remaining digraph is still big, so we can find a long directed path by using \Cref{thm:digraph-dirac}. Then we connect branch vertices and the directed paths by using the vertices in the connector. Then we obtain a directed subdivision and a small number of vertices. The remaining vertices can be absorbed into the absorbing path, thus we can obtain the desired spanning directed subdivision.

The next lemma states that we can construct an absorbing path.

\begin{lemma}\label{lem:absorber}
    Let
    $$0 < \frac{1}{n} \ll \beta \ll \alpha \ll \ve < 1.$$

    Let $D$ be an $n$-vertex graph with $\delta^0(D) \geq \left(\frac{1}{2} + \ve \right)n.$ Then there is a directed path $A$ in $D$ satisfying $|V(A)| \leq \alpha n$ such that for every $U\subseteq V(G)\setminus A$ with $|U| \leq \beta n$, there is a directed path $P$ in $D$ on the vertex set $V(A)\cup U$ that has the same endpoints with the path $A.$
\end{lemma}

\begin{proof}[Proof of \Cref{lem:absorber}]
    For a vertex $u\in V(D)$, we say a tuple $(v, w)\in V(D)\times V(D)$ is a \emph{good tuple} for $u$ if $v \neq w$ and $\overrightarrow{vu}, \overrightarrow{uw}, \overrightarrow{vw}\in A(D).$ We claim that every vertex of $D$ has many good tuples.
    
    \begin{claim}\label{clm:many-transitive-trianlge}
        For every $u\in V(D)$, there are at least $4\ve^2 n^2$ good tuples for $u.$
    \end{claim}

    \begin{claimproof}[Proof of \Cref{clm:many-transitive-trianlge}]
        Let $U$ be a subset of $V(D)$ with size at least $\left(\frac{1}{2} + \ve \right)n$ and $U'$ be a subset of $U$ with size at least $2\ve n.$ Then for every $u'\in U'$, we have $|N^+_D(u')\cap U| \geq 2\ve n.$ Thus, the inequality $|A_D[U', U]| = \sum_{u'\in U'} |N^+_D(u')\cap U| \geq 4\ve^2 n^2$ holds. Take $U = N^+(u)$ and $U' = U\cap N^-(u).$ Then we have $|U| \geq \left(\frac{1}{2} + \ve \right)n$ and $|U'| \geq 2\ve n.$ Then for every $\overrightarrow{vw}\in A_D[U', U]$, the tuple $(v, w)$ is a good tuple for $u.$ Since $A_D[U', U] \geq 4\ve^2 n^2$, the claim holds.  
    \end{claimproof}
    
    Let $X = Y =V(D)$ and let $\calP = \{(u, (v, w)) \in X\times Y^2: (v, w) \text{ is a good tuple for } u\}.$ Then by \Cref{clm:many-transitive-trianlge}, $\calP$ is an $\left(n, 2, 4\ve^2 \right)$-tuple system on $(X, Y).$ Thus, by \Cref{lem:ntd-absorbing}, there is a set $\calY\subseteq V(D)^2$ with size at most $\frac{\alpha n}{3}$ such that elements in $\calY$ are pairwise tuple-disjoint and for every $u\in V(D)$, we have $|\{(v, w)\in \calY: (u, (v, w))\in \calP\}| \geq \beta n.$ Let $\calY = \{(v_1, w_1), \dots, (v_{\ell}, w_{\ell})\}.$ Since we have $d_D(x, y) \geq 2\ve n \geq \alpha n \geq 3|\calY|$ for every pair of distinct vertices $(x, y)$, we can greedily choose distinct vertices $x_1, \dots, x_{\ell-1}$ which are disjoint from $\{v_1,\dots, v_{\ell}, w_1, \dots, w_{\ell}\}$ such that $\overrightarrow{w_ix_i}, \overrightarrow{x_iv_{i+1}} \in A(D)$ for each $i \in [\ell - 1].$ This implies that there is a directed path $A = v_1w_1x_1v_2w_2x_2 \cdots x_{\ell - 1}v_{\ell}w_{\ell}$ in $D.$ 
    
    We now show that $A$ is a desired directed path. Note that $|V(A)| \leq 3|\calY| \leq \alpha n.$ Let $U \subseteq V(D)\setminus V(A)$ be a vertex set with size at most $\beta n.$ By our choice of $\calY$, for each $u\in U$, there are at least $\beta n$ pairs of $(v_i, w_i)$ such that $\overrightarrow{v_iu}, \overrightarrow{uw_i} \in A(D).$ Thus, we can greedily absorb all the vertices of $U$ by replacing the arc $\overrightarrow{v_iw_i}$ with $\overrightarrow{v_i u}, \overrightarrow{uw_i}$ and still maintain the same endpoints $v_1$ and $w_{\ell}.$ This proves the lemma.  
\end{proof}

The remaining proof ingredient is the connector. The next lemma shows that we also have a connector.

\begin{lemma}\label{lem:connect}
    Let
    $$0 < \frac{1}{n} \ll \beta \ll \alpha \ll \ve < 1.$$

    Let $D$ be an $n$-vertex digraph with $\delta^0(D) \geq \left(\frac{1}{2} + \ve \right)n.$ Then there is a set $R \subseteq V(D)$ with size $|R| \leq \alpha n$ such that for any distinct vertices $u, v \in V(D)$, we have $|N_D(u, v)\cap R| \geq \beta n.$
\end{lemma}

\begin{proof}[Proof of \Cref{lem:connect}]
    Let $X = V(D)\times V(D)\setminus \{(v, v): v\in V(D)\}$ and let $Y = V(D).$ Let $\calP  = \{((u, v), w)\in X\times Y: w\in N_D(u, v)\}.$ By the minimum semi-degree condition, for every distinct pair of vertices $(u, v)$ of $V(G)$, we have $d_D(u, v) \geq 2\ve n.$ This implies that $\calP$ is an $(n, 1, 2\ve)$-tuple system on $(X, Y).$ By \Cref{lem:ntd-absorbing}, there is $\calY \subseteq Y$ with $|\calY| \leq \alpha n$ and for every $(u, v) \in X$, we have $|\{w\} \in \calY: ((u, v), w) \in \calP| \geq \beta n.$ This implies that $R = \calY$ is the desired set.
\end{proof}

We now prove \Cref{thm:main} by using \Cref{lem:absorber,lem:connect}. Our proof produces a spanning $H$-subdivision that contains a very long directed path, but it is easy to verify that we can also obtain an almost balanced spanning $H$-subdivision by slightly modifying the argument.

\begin{proof}[Proof of \Cref{thm:main}]
    Let $$0 < \frac{1}{C} \ll \gamma \ll \beta \ll \alpha \ll \ve < 1.$$
    Let $D$ be an $n$-vertex digraph with $\delta^0(D) \geq \left(\frac{1}{2} + \ve \right)n$ where $n\geq Cm.$ Let $H$ be an $m$-arc digraph without isolated vertices. By \Cref{lem:absorber}, there is a path $A$ in $D$ that satisfies the \Cref{lem:absorber} with parameters $\alpha$ and $\beta.$ Let $w_1, w_2$ be the endpoints of the path $A.$ Then we have $|V(A)| \leq \alpha n$ and the following holds.
    For every $U\subseteq V(D)\setminus V(A)$ of size at most $\beta n$, there is a directed path $P$ in $D$ which starts at $w_1$ and ends at $w_2$ where $V(P) = V(A) \cup U.$ 

    Let $D_1 \defeq D - V(A) + \{w_1, w_2\}.$ Then we have $\delta^0(D_1) \geq \left(\frac{1}{2} + \frac{\ve}{2} \right)n.$ By \Cref{lem:connect}, there is a set $R \subseteq V(D_1)$ that satisfies the conclusions of \Cref{lem:connect} with parameters $\frac{\ve}{2}$, $\beta$, $2\gamma.$ If $R \cap \{w_1, w_2\} \neq \emptyset$, then we simply remove $w_1$ and $w_2$ from $R.$ Then $R$ is disjoint from $\{w_1, w_2\}$ and we have $|R| \leq \beta n$ and for any two distinct vertices $x, y \in V(D_1)$, we have $|N_D(x, y)\cap R| \geq \gamma n.$ Let $D_2\defeq D - (V(A)\cup R).$ Then we have $\delta^0(D_2) \geq \left(\frac{1}{2} + \frac{\ve}{4}\right) n.$

    Let $V(H) = \{x_1, \dots, x_k\}.$ We arbitrary choose vertices $v_1, \dots v_k$ in $D_2.$ Our plan is embed $x_i$ to $v_i$ as a branch vertex for each $i\in [k].$ Let $D_3\defeq D_2 - \{v_1,\dots v_k\}.$ Note that $k \leq 2m \leq \frac{\ve}{4}Cm \leq \frac{\ve}{4}n.$ Thus, we have $\delta^0(D_3) \geq \frac{n}{2}.$ By \Cref{thm:digraph-dirac}, the digraph $D_3$ has a directed Hamiltonian path that starts from $u_1$ and ends at $u_2.$ Without loss of generality, we may assume that $\overrightarrow{x_1x_2}\in A(H).$ By our choice of $R$, there are three distinct vertices $z_1, z_2, z_3 \in R$ such that $\overrightarrow{v_1z_1}, \overrightarrow{z_1u_1}, \overrightarrow{u_2z_2}, \overrightarrow{z_2w_1}, \overrightarrow{w_2z_3}, \overrightarrow{z_3v_2} \in A(D).$ This implies that there is a directed path $Q$ starts from $v_1$ and ends at $v_2$ such that $Q$ contains the directed path $A$, all vertices of $D_3$, and exactly three elements of $R.$ 

    Let $A(H)\setminus \{\overrightarrow{x_1x_2}\} = \{e_1, \dots, e_{m-1}\}.$ Since $\gamma n \geq \gamma C m \geq m + 2$, for each $i\in [m-1]$ where $e_i = \overrightarrow{x_ax_b}$, we can greedily choose distinct vertices $z_{e_i}\in R,$ disjoint from $\{z_1, z_2, z_3\},$ such that $\overrightarrow{v_a z_i}, \overrightarrow{z_iv_b} \in E(G).$ Thus, we can obtain a subdivision of $H$ that contains the directed path $A$, and the remaining vertices are contained in $R.$ Since $|R| \leq \beta n$, by our choice of $A$, we can absorb the remaining vertices, hence we obtain a spanning subdivision of $H$ in $D.$ This completes the proof.

\end{proof}


\printbibliography
\end{document}